%% file: lindelofify.tex
\documentclass[a4paper]{amsart}
\def\TeXRoot{preamble}
\input{\TeXRoot/input_pdflatex_mini.sty}
\input{settings.sty}
\begin{document}
\maketitle
\begin{abstract}
  In the present paper, we prove that a topological space admits a functorial Lindel\"ofification if and only if its realcompactification is Lindel\"of.
  To investigate the functorial Lindel\"ofifiability of a topological space, for each topological property \(\sfP\), we introduce the notion of ``\textbf{functorial \(\sfP\)-ification}'' and give an explicit construction of the functorial \(\sfP\)-ification.
  Moreover, for a discrete space \(X\), we discuss the functorial \(|X|\)-Lindel\"ofifiability of \(X\) and study relationships with properties of the cardinal \(|X|\).
  Finally, we apply our results concerning functorial \(\kappa\)-Lindel\"ofifiability (for some cardinal \(\kappa\)) to the space of ordinals and construct several functorial \(\kappa\)-Lindel\"ofifiable spaces.
\end{abstract}
\input{section/0_intro.tex}
\input{section/1_func_lind.tex}
\input{section/2_cech_cplt.tex}
\input{section/3_real_cpt.tex}
\input{section/4_kappa_Lind.tex}
\input{section/5_example.tex}

\input{section/z_biblio.tex}
\end{document}

%% file: section/0_intro.tex
\section*{Introduction}

Throughout the present paper, we always suppose that \textbf{topological spaces are completely regular and Hausdorff}.
Let \(X\) be a topological space.
Our interest in the present paper is the property of the \textit{functorial} Lindel\"ofification of \(X\), where the functorial Lindel\"ofification of \(X\) is an extension space \(X\to X'\) such that for any continuous map \(f: X\to Y\), if \(Y\) is Lindel\"of, then \(f\) can be extended uniquely to a continuous map \(X'\to Y\).
Although the Stone-\v{C}ech compactification always exists for any completely regular Hausdorff space of \(X\), the functorial Lindel\"ofification may not exist.
To investigate the functorial Lindel\"ofifiability of a topological space, for each topological property \(\sfP\), we introduce the notion of ``functorial \(\sfP\)-ification'' and study the structure of the functorial \(\sfP\)-ification of a topological space.

In \cite{AHST}, F. Azarpanah, A. A. Hesari, A. R. Sarehi, and A. Taherifar constructed several Lindel\"of-like extension spaces of \(X\) in the Stone-\v{C}ech compactification \(\beta X\).
They, in detail, studied their extension spaces from the point of view of the relationship between algebraic properties of the rings of continuous (resp. bounded continuous) functions \(C(X)\) (resp. \(C^*(X)\)) on \(X\) and topological properties of \(X\) or \(\beta X\).
However, their Lindel\"of-like extension spaces do not have a suitable functoriality.
Therefore, by contrast, in the present paper, we discuss what properties can be concluded from the abstract functorial Lindel\"ofifiability.
In particular, our interest is the existence of functorial \(\sfP\)-ification and topological characterization of the functorial \(\sfP\)-ifiability.


\begin{acknowledgement*}
  The author would like to thank significantly K. Takahashi.
  The present investigation started with his small question about Lindel\"ofification.
  The author would also like to thank A. Yamashita for helpful conversations on perfect mappings (at Tsubakiya coffee, Roppongi) and Y. Ishiki for helpful advice on properties of locally compact \(\sigma\)-compact spaces.
\end{acknowledgement*}

\begin{notation*}
  We shall write \(\Top\) for the category of \textbf{completely regular Hausdorff} topological spaces and continuous functions.
  We shall write \(\R\) for the topological field of real numbers.

  Let \(X\) be a set.
  We shall write \(|X|\) for the cardinality of \(X\).
  We shall write \(2^X\) for the power set of \(X\).
  For any family of sets \(\mcF\), we shall write \(\bigcup \mcF\) for the union of \(\mcF\).

  Let \(X\) be a topological space.
  We shall write \(C(X)\) (resp. \(C^*(X)\)) for the rings of continuous (resp. bounded continuous) functions on \(X\).
  We shall write \(\beta X\) for the Stone-\v{C}ech compactification of \(X\).
  Then, the natural restriction morphism of rings \(C(\beta X)\xto{(-)|_X} C(X)\) induces an isomorphism of rings \(C(\beta X)\xto[\sim]{(-)|_X} C^*(X)\).
\end{notation*}

%% file: section/1_func_lind.tex
\Section[Functorial P-ification]{Functorial \(\sfP\)-ification}

Let \(\sfP\subset \Top\) be a full subcategory.
In this section, we study the fundamental properties of the \textit{\Pify} and construct it as a subspace of the Stone-\v{C}ech compactification (cf. \cref{def: func P-ify construct}).

Let us start to define the \Pify, which is an analogue of the Stone-\v Cech compactification.

\begin{definition}
  Let \(\sfP\subset\Top\) be a full subcategory.
  \begin{enumerate}
    \item
    Let \(X\) be a topological space.
    We shall say that \(X\) \textbf{admits a \Pify} \(i_X: X\to \nu_{\sfP}X\) if \(\nu_{\sfP}X\) belongs to \(\sfP\), and, moreover, for any object \(Y\in \sfP\) and any continuous map \(f:X\to Y\), there exists a unique continuous map \(\tilde{f}:\nu_{\sfP}X\to Y\) such that \(f = \tilde{f}\circ i\).
    If \(X\) admits a \Pify\ \(i_X:X\to \nu_{\sfP}X\), then we shall say that \(X\) is \textbf{\Piable}.
    \item
    Let \(f:X\to Y\) be a continuous map between topological spaces that admit \Pifys.
    Write \(\nu_{\sfP}f:\nu_{\sfP}X\to \nu_{\sfP}Y\) for the unique continuous map such that \(i_Y\circ f = \nu_{\sfP}f\circ i_X\).
    \item
    Write \(\Pfifiable\subset \Top\) for the full subcategory determined by the topological spaces that admits a \Pify.
  \end{enumerate}
\end{definition}

\begin{definition}
  We shall write \(\Cpt\subset \Top\) for the full subcategory of compact Hausdorff topological spaces (where we note that, by Urysohn's lemma, a compact Hausdorff topological space is completely regular).
\end{definition}

\begin{remark}
  If \(X\) is a topological space that belongs to \(\sfP\), then the identity morphism \(\id_X: X\to X = \nu_{\sfP}X\) satisfies property of the \Pify.
  Hence, it holds that \(\sfP\subset \Pfifiable\).
  Moreover, the assignment \(\nu_{\sfP}: \Pfifiable \to \sfP\), \(X\mapsto \nu_{\sfP}X\), \(f\mapsto \nu_{\sfP}f\) is a left adjoin functor of the inclusion functor \(\sfP \subset \Pfifiable\) such that \(\nu_{\sfP}|_{\sfP} = \id_{\sfP}\).
\end{remark}

\begin{remark}
  Assume that \(\sfP=\Cpt\subset \Top\).
  By the uniqueness of a left adjoint functor of the inclusion functor \(\sfP\subset \sfP^{\nu}\), then the \Pify\ functor \(\nu_{\Cpt}\) is isomorphic to the \SCCpt\ functor \(\beta:\Top\to \Cpt\).
\end{remark}

\begin{lemma}\label{lem: func Lind is a sub of SC cpt}
  Let \(\sfP\subset \Top\) be a full subcategory such that \(\Cpt\subset \sfP\) and \(X\) a topological space that admits a \Pify\ \(i_X:X\to \nu_{\sfP}X\).
  Then the continuous map \(\beta i_X: \beta X\to \beta\nu_{\sfP}X\) is a homeomorphism.
  In particular, the continuous map \(i_X: X \to \nu_{\sfP}X\) may be regarded as the inclusion map \(X\subset \nu_{\sfP}\) between subspaces of \(\beta X\).
\end{lemma}

\begin{proof}
  Write \(j_{(-)}: (-)\to \beta(-)\) for the natural embedding to the \SCCpt\ of \((-)\).
  Since \(\beta X\in \Cpt\subset \sfP\), it follows from the universality of the \Pify\ \(\nu_{\sfP}X\) that there exists a unique morphism \(j: \nu_{\sfP}X\to \beta X\) such that \(j_X = j\circ i_X\).
  By the universality of the \SCCpt\ of \(\nu_{\sfP}X\), there exists a morphism \(g:\beta\nu_{\sfP}X\to \beta X\) such that \(j=g\circ j_{\nu_{\sfP}X}\):
  \[\begin{tikzcd}
    {X}
      \ar[r, "{j_X}"] \ar[d, "{i_X}"'] &[1cm]
    {\beta X}
      \ar[d, "{\beta i_X}"'] \\
    {\nu_{\sfP} X}
      \ar[r, "{j_{\nu_{\sfP}X}}"] \ar[ru, "{j}"] &
    {\beta\nu_{\sfP}X}
      \ar[u, "g"', bend right=20pt].
  \end{tikzcd}\]
  By the functoriality of the operation \(\beta(-)\), it holds that \(\beta i_X \circ j_X = j_{\nu_{\sfP}X}\circ i_X\).
  Hence, it holds that
  \[g\circ \beta i_X \circ j_X = g\circ j_{\nu_{\sfP}X}\circ i_X = j\circ i_X = j_X.\]
  By the universality of the \SCCpt\ of \(X\), this implies that \(g\circ \beta i_X = \id_{\beta X}\).
  Moreover, since \(\beta i_X \circ j_X = j_{\nu_{\sfP}X}\circ i_X\), it holds that
  \((\beta i_X\circ j)\circ i_X = \beta i_X \circ j_X = j_{\nu_{\sfP}X} \circ i_X\).
  Since \(\beta\nu_{\sfP}X\in \Cpt\subset \sfP\), it follows from the universality of the \Pify\ of \(X\) that \(\beta i_X \circ j = j_{\nu_0X}\).
  Hence, it holds that
  \[\beta i_X \circ g \circ j_{\nu_{\sfP}X} = \beta i_X \circ j = j_{\nu_{\sfP}X}.\]
  By the universality of the \SCCpt\ of \(\nu_{\sfP}X\), this implies that \(\beta i_X \circ g = \id_{\beta\nu_{\sfP}X}\).
  Thus, in particular, the natural continuous map \(\beta i_X\) is a homeomorphism.
  This completes the proof of \cref{lem: func Lind is a sub of SC cpt}.
\end{proof}

Next, we construct explicitly an extension of \(X\) in \(\beta X\) that represents the \Pify.

\begin{definition}\label{def: func P-ify construct}
  Let \(\sfP\subset \Top\) be a full subcategory such that \(\Cpt\subset \sfP\) and \(X\) a topological space.
  We shall write
  \[
    \tilde{\nu}_{\sfP} X \dfn \bigcap
    \left\{
      \beta f^{-1}(Y)
      \,\middle|\,
      \text{\(f:X\to Y\) is a continuous map such that \(Y\in\sfP\)}
    \right\} \,\subset\, \beta X.
  \]
\end{definition}

\begin{lemma}\label{lem: P and Q ify}
  Let \(\sfP\subset \mathsf{Q}\subset \Top\) be full subcategories and \(X\) a topological space.
  Then, it holds that \(\tilde{\nu}_{\mathsf{Q}} X \subset \tilde{\nu}_{\sfP}X\).
\end{lemma}

\begin{proof}
  \Cref{lem: P and Q ify} follows immediately from the definition of \(\tilde{\nu}_{(-)}(?)\) and the inclusion relation \(\sfP\subset \mathsf{Q}\).
\end{proof}

\begin{lemma}\label{lem: nu X and naive}
  Let \(\sfP\subset \Top\) be a full subcategory such that \(\Cpt\subset \sfP\) and \(X\) a \Piable\ topological space.
  Then, it holds that \(\nu_{\sfP}X = \tilde{\nu}_{\sfP}X\).
\end{lemma}

\begin{proof}
  If we write \(i_X:X\inj \nu_{\sfP}X\) for the inclusion map, then, since \(\nu_{\sfP}X\in \sfP\), it follows from  \cref{lem: func Lind is a sub of SC cpt} that \(\tilde{\nu}_{\sfP}X\subset \beta i_X^{-1}(\nu_{\sfP}X) = \nu_{\sfP}X\).

  Next, we prove that \(\nu_{\sfP}X\subset \tilde{\nu}_{\sfP}X\).
  Let \(f:X\to Y\) be a continuous map such that \(Y\in \sfP\).
  Then, by the universality of the \Pify, there exists a unique morphism \(\tilde{f}:\nu_{\sfP}X\to Y\) such that \(\tilde{f}|_X = f\).
  By \cref{lem: func Lind is a sub of SC cpt}, it holds that \(\beta f = \beta \tilde{f}\).
  Thus, it holds that \(\nu_{\sfP}X\subset \beta\tilde{f}^{-1}(Y) = f^{-1}(Y)\).
  By allowing \(Y\) to vary over \(\sfP\), we conclude that \(\nu_{\sfP}X\subset \tilde{\nu}_{\sfP} X\).
  This completes the proof of \cref{lem: nu X and naive}.
\end{proof}


\begin{lemma}\label{lem: nu X in SC Lind}
  Let \(\sfP\subset \Top\) be a full subcategory such that \(\Cpt\subset \sfP\) and \(X\) a topological space.
  Then, \(X\) is \Piable\ if and only if \(\tilde{\nu}_{\sfP}X\) belongs to \(\sfP\).
\end{lemma}

\begin{proof}
  If \(X\) is \Piable, then, by \cref{lem: nu X and naive}, it holds that \(\tilde{\nu}_{\sfP}X = \nu_{\sfP}X \in \sfP\).
  Conversely, if \(\tilde{\nu}_{\sfP}X \in \sfP\), then one can verify immediately that \(\tilde{\nu}_{\sfP}X\) satisfies the required universality of the \Pify\ of \(X\).
  This completes the proof of \cref{lem: nu X in SC Lind}.
\end{proof}

%% file: section/2_cech_cplt.tex
\Section[Cech completion]{Almost Compactification}

In this section, we consider the case where any \Piable\ space belongs to \(\sfP\).

\begin{definition}
  Let \(X\) be a topological space.
  We shall say that \(X\) is \textbf{almost compact} if \(\beta X \setminus X\) is of cardinality at most one.
  Write \(\AlmCpt\subset \Top\) for the full subcategory determined by almost compact spaces.
\end{definition}

\begin{remark}\label{rmk: beta minus one pt is alm cpt}
  For any topological space \(X\) and any subspace \(X\subset Y\subset \beta X\), the natural morphism \(\beta X\to \beta Y\) is a homeomorphism.
  Indeed, any bounded continuous function \(f:Y\to [a,b]\subset \R\) can be extended uniquely to a continuous function \(\beta f|_Y:\beta X \to [a,b]\subset \R\).
  Thus, in particular, for any \(p\in \beta X\setminus X\), \(\beta X\setminus \{p\}\) is almost compact.
\end{remark}

Note that for any topological space \(X\), it follows immediately that
\[X = \bigcap_{p\in \beta X \setminus X}(\beta X \setminus \{p\}).\]
By \cref{rmk: beta minus one pt is alm cpt}, the equality displayed above leads us to the following proposition:

\begin{proposition}\label{prop: cech cplt sfP}
  Let \(\sfP\subset \Top\) be a full subcategory such that \(\AlmCpt\subset \sfP\) and \(X\) be a topological space.
  Then, \(X\) is \Piable\ if and only if \(X\) belongs to \(\sfP\).
\end{proposition}

\begin{proof}
  Since \(\AlmCpt \subset \sfP\), it follows from \cref{rmk: beta minus one pt is alm cpt} and the definition of \(\tilde{\nu}_{\sfP}\) that
  \[
    X\subset \tilde{\nu}_{\sfP}X \subset \bigcap_{p\in \beta X \setminus X}(\beta X \setminus \{p\}) = X.
  \]
  Hence, it holds that \(X = \tilde{\nu}_{\sfP}X\).
  Thus, \cref{prop: cech cplt sfP} follows immediately from \cref{lem: nu X in SC Lind}.
\end{proof}

\begin{remark}\label{rmk: pseudocompact realcompact}
  One can easily prove that an almost compact topological space is pseudocompact (cf. \cref{def: pseudocompact}) and locally compact.
  Hence, an almost compact space is \v{C}ech complete.
  In particular, by \cref{prop: cech cplt sfP}, if \(\sfP\) is equal to the full subcategory determined by these topological properties, then \(\sfP = \sfP^{\nu}\).
\end{remark}

%% file: section/3_real_cpt.tex
\Section[real compactification]{Realcompactification}

In the present section, we apply the theory developed in \cref{Functorial P-ification} to the case where the category \(\sfP\subset \Top\) consists of realcompact topological spaces and characterize functorial Lindel\"ofifiability as a topological property of realcompactification (cf. \cref{cor: func lind}).

\begin{definition} \
  \begin{enumerate}
    \item We shall write \(\RCpt\subset \Top\) for the full subcategory determined by the realcompact topological spaces (cf. \cref{remark: real compact property}).
    \item For any topological space \(X\), we shall write \(\up X\) for the realcompactification of \(X\) (cf. \cref{remark: real compact property} or \cite[Section 8.4]{GJ-rings-of-conti}).
  \end{enumerate}
\end{definition}

In the present section, we shall mainly be concerned with the situation that \(\R\in \sfP\subset \RCpt\).
Let us recall some fundamental relationships between realcompactifications and rings of continuous functions.

\begin{remark}
  \label{remark: real compact property}
  Let \(X\) be a topological space.
  Recall that \(X\) is \textbf{realcompact} if there exist a cardinal \(\kappa\) and a closed embedding \(X\inj \R^{\kappa}\).
  A point \(x\in\beta X\) of the \SCCpt\ of \(X\) is \textbf{real} if any continuous map \(X\to \R\) can be extended to a continuous map \(X\cup\{x\}\to \R\).
  Here, we note that any point of \(X\) is a real point.
  Then, it is a well-known fact that \(X\) is realcompact if and only if any real point \(x\in \beta X\) belongs to \(X\subset \beta X\).

  Note that the realcompactification \(\up X\) is defined as the set of real points of \(X\) in \(\beta X\).
  By the characterization of realcompactness mentioned as above, \(\up X\) is automatically realcompact.
  Hence, it holds that
  \begin{equation}
    \label{ass: real cpt ext remark: real compact property}
    \up X = \bigcap_{f\in C(X)} \beta f^{-1}(\R).
    \tag{\(\dagger\)}
  \end{equation}
  In particular, the realcompactification \(\up X\) is the minimal realcompact subspace of \(\beta X\).
  Thus, we conclude from these discussion that the natural restriction morphism \((-)|_X : C(\up X) \to C(X)\) is an isomorphism of rings.

  A real point \(x\in \beta X\) of \(X\) can be characterized as a point \(x\in \beta X\) such that the residue field of the ring \(C(X)\) at the maximal ideal determined by \(x\in \beta X\) is isomorphic to \(\R\).
  Hence, the topological space \(\up X\) may be reconstructed as the set of maximal ideals of \(C(X)\) whose residue field is isomorphic to \(\R\), together with the topology induced by the \textbf{Zariski topology} of the prime spectrum of the ring \(C(X)\) (where we note that, by Urysohn's lemma, the topology induced by the Zariski topology coincides with the subspace topology of \(\beta X\), cf., e.g., \cite[Problem 1.26 (ii)]{AM}).
  Thus, if \(Y\) is an another topological space, and \(f^*: C(Y)\to C(X)\) is a morphism of rings, then \(f^*\) induces a unique continuous map \(\up f: \up X\to \up Y\) such that the following diagram commutes:
  \begin{equation}
    \label{eq: commutes nu f remark: real compact property}
    \begin{tikzcd}
      {C(\up Y)}
        \ar[r, "{(-)|_Y}", "\sim"']
        \ar[d, "{(-)\circ \up f}"'] &
      {C(Y)}
        \ar[d, "{f^*}"] \\
      {C(\up X)}
        \ar[r, "{(-)|_X}", "\sim"'] &
      {C(X)}.
    \end{tikzcd}
    \tag{\(\ddagger\)}
  \end{equation}
\end{remark}

\begin{remark}[Structure of the Realcompactification]
  \label{rmk: structure of realcompactification}
  It is a well-known fact that a locally compact Lindel\"of space \(X\) admits a \textbf{perfect map} (i.e., a continuous closed map whose fibers are compact) \(X\to \R\) (cf., e.g., \cite[Proposition 2.4]{Ishiki-san-extending-proper-metric}).
  This follows easily from Frol\'ik's theorem (cf. \cite{Frolik-product-of-paracompacts}, where we note that a Lindel\"of locally compact space is \v{C}ech-complete and paracompact), but we can directly prove this as follows:
  Since \(X\) is locally compact, the one-point compactification \(\alpha X\) of \(X\) is Hausdorff.
  Write \(\infty\in\alpha X\setminus X\) for the unique point.
  By Urysohn's theorem, for any \(x\in X\), there exists a continuous map \(f_x:\alpha X \to [0,1]\) such that \(f_x(\infty) = 0\), and \(f_x(x) = 1\).
  Since \(X\) is Lindel\"of, there exists a countable sequence \(x_0,x_1,\cdots\in X\) such that \(X = \bigcup_{n\in \N}f_{x_n}^{-1}((0,1])\).
  Then, the continuous map \(f\dfn \sum_{n\in \N}\frac{1}{2^n}f_{x_n}\) satisfies that \(f^{-1}(0) = \{\infty\}\).
  This implies that the restriction \(f|_X:X\to \R\) is perfect.

  Thus, if a full subcategory \(\{\R\}\cup \Cpt \subset \sfP\subset \Top\) is contained in the class ``locally compact Lindel\"of spaces'', then for any subset \(X\subset Z\subset \beta X\), the following assertions are equivalent:
  \begin{enumerate}
    \item There exists a continuous map \(f:X\to \R\) such that \(Z = \beta f^{-1}(\R)\).
    \item There exist an object \(Y\in \sfP\) and a continuous map \(f:X\to Y\) such that \(Z = \beta f^{-1}(Y)\).
    \item \(Z\) is locally compact Lindel\"of.
  \end{enumerate}
  In particular, it holds that
  \[\up X = \bigcap\left\{X\subset Y\subset \beta X\,\middle|\, \text{\(Y\) is Lindel\"of locally compact}\,\right\} \, \subset\, \beta X. \]
  This implies that for any perfect map \(f:X\to Y\), if \(Y\) is realcompact, then \(X\) is also realcompact.
\end{remark}

\begin{lemma}\label{cor: tilde nu P sub nu X}
  Let \(\Cpt\subset \sfP\subset \Top\) be a full subcategory and \(X\) a topological space.
  Then, the following assertions hold:
  \begin{assertion}
    \item \label{ass: scpt cor: tilde nu P sub nu X}
    Assume that \(\R \in \sfP\).
    Then, it holds that \(\tilde{\nu}_{\sfP}X\subset \up X\).
    \item \label{ass: rcpt cor: tilde nu P sub nu X}
    Assume that \(\sfP\subset \RCpt\).
    Then, it holds that \(\up X\subset \tilde{\nu}_{\sfP}X\).
  \end{assertion}
  In particular, if \(\{\R\}\cup\Cpt\subset \sfP\subset \RCpt\), then it holds that \(\up X = \tilde{\nu}_{\sfP}X\).
\end{lemma}

\begin{proof}
  Since \(\R\in \sfP\), \cref{ass: scpt cor: tilde nu P sub nu X} follows immediately from \cref{ass: real cpt ext remark: real compact property} in \cref{remark: real compact property} and the definition of \(\tilde{\nu}_{\sfP}X\).
  Let \(f:X\to Y\) be a continuous map such that \(Y\) is realcompact.
  Then, \(\beta f^{-1}(Y)\) is also realcompact (cf. \cref{rmk: structure of realcompactification}).
  Thus, \cref{ass: rcpt cor: tilde nu P sub nu X} follows immediately from the fact that \(\up X\) is the minimal realcompact subspace of \(\beta X\) such that \(X\subset \up X\) (cf. the sentence after the \cref{ass: real cpt ext remark: real compact property} in \cref{remark: real compact property}).
  This completes the proof of \cref{cor: tilde nu P sub nu X}.
\end{proof}

We then conclude our main result in the present paper.

\begin{theorem}\label{cor: func P-ifiable iff nu X in P}
  Let \(\Cpt\subset \sfP\subset \Top\) be a full subcategory such that \(\R\in \sfP\subset \RCpt\) and \(X\) a topological space.
  Then, the following assertions are equivalent:
  \begin{assertion}
    \item \label{ass: Piable cor: func P-ifiable iff nu X in P}
    \(X\in \sfP^{\nu}\).
    \item \label{ass: nuX in P cor: func P-ifiable iff nu X in P}
    \(\up X\in \sfP\).
    \item \label{ass: CX isom cor: func P-ifiable iff nu X in P}
    There exists an object \(Y\in \sfP\) such that \(C(X)\) is isomorphic as a ring to \(C(Y)\).
  \end{assertion}
  Moreover, if \cref{ass: CX isom cor: func P-ifiable iff nu X in P} holds, then the \Pify\ of \(X\) coincides with \(Y\).
\end{theorem}

\begin{proof}
  Equivalent ``\ref{ass: Piable cor: func P-ifiable iff nu X in P} \(\Leftrightarrow\) \ref{ass: nuX in P cor: func P-ifiable iff nu X in P}'' follows immediately from \cref{lem: nu X in SC Lind} and \cref{cor: tilde nu P sub nu X}.
  Equivalent ``\ref{ass: nuX in P cor: func P-ifiable iff nu X in P} \(\Leftrightarrow\) \ref{ass: CX isom cor: func P-ifiable iff nu X in P}'' follows immediately from the fact that the realcompactification of \(X\) can be reconstructed as a set of maximal ideals of \(C(X)\) whose residue field is \(\R\), together with the Zariski topology (cf. the above sentence of the diagram \eqref{eq: commutes nu f remark: real compact property} in \cref{remark: real compact property}).
  The last assertion also follows formally from this fact.
  This completes the proof of \cref{cor: func P-ifiable iff nu X in P}.
\end{proof}

\begin{corollary}\label{cor: sfP nu cap RCpt is sfP}
  Let \(\Cpt\subset \sfP\subset \Top\) be a full subcategory such that \(\R\in \sfP\subset \RCpt\).
  Then, it holds that \(\sfP^{\nu}\cap \RCpt = \sfP\).
\end{corollary}

\begin{proof}
  Since \(\sfP\subset \sfP^{\nu}\), it holds that \(\sfP\subset \sfP^{\nu}\cap\RCpt\).
  Let \(X\) be \Piable\ realcompact topological space.
  Then, by the equivalence ``\ref{ass: Piable cor: func P-ifiable iff nu X in P} \(\Leftrightarrow\) \ref{ass: nuX in P cor: func P-ifiable iff nu X in P}'' in \cref{cor: func P-ifiable iff nu X in P}, it holds that \(X = \up X\in \sfP\).
  In particular, \(\sfP^{\nu}\cap \RCpt \subset \sfP\).
  This completes the proof of \cref{cor: sfP nu cap RCpt is sfP}.
\end{proof}

Next, we prove that \(\sfP\subsetneq \sfP^{\nu}\) if \(\R \in \sfP \subset \RCpt\).

\begin{definition}\label{def: pseudocompact}
  We shall write \(\PsCpt\subset \Top\) for the full subcategory consisting of pseudocompact topological spaces.
  Recall that a topological space \(X\) is \textbf{pseudocompact} if any continuous map \(X\to \R\) is bounded, i.e., \(C^*(X) = C(X)\).
\end{definition}

By the definition of the notion of a pseudocompact space, a topological space \(X\) is pseudocompact if and only if \(\up X = \beta X\).
Thus, the following assertion holds:

\begin{proposition}\label{prop: sfP nu pscpt}
  Let \(\Cpt\subset \sfP\subset \Top\) be a full subcategory such that \(\R\in \sfP\subset \RCpt\).
  Then, it holds that \(\PsCpt \subset \sfP^{\nu}\).
  In particular, it holds that \(\sfP\subsetneq \sfP^{\nu}\).
\end{proposition}

\begin{proof}
  Since \(\Cpt\subset \sfP\), it follows from \cref{cor: func P-ifiable iff nu X in P} and the definition of the notion of a pseudocompact space that \(\PsCpt \subset \sfP^{\nu}\).
  Moreover, it follows immediately from the definition of the notion of a pseudocompact realcompact space that \(\PsCpt \cap \RCpt = \Cpt\).
  Since \(\Cpt\subsetneq \PsCpt\), this implies that \(\sfP\subsetneq \sfP^{\nu}\).
  This completes the proof of \cref{prop: sfP nu pscpt}.
\end{proof}

At the end of the present section, we apply \cref{cor: func P-ifiable iff nu X in P} to the category of Lindel\"of spaces.

\begin{definition}
  We shall write \(\Lind\subset\Top\) for the full subcategory determined by the Lindel\"of spaces.
\end{definition}

\begin{corollary}\label{cor: func lind}
  Let \(X\) be a topological space.
  Then, the following assertions hold:
    \begin{assertion}
      \item \(X\) is functorial Lindel\"ofifiable.
      \item The realcompactification \(\up X\) of \(X\) is Lindel\"of.
      \item There exists a Lindel\"of space \(Y\) such that \(C(X)\) is isomorphic as a ring to \(C(Y)\).
    \end{assertion}
\end{corollary}

\begin{proof}
  \Cref{cor: func lind} follows immediately from \cref{cor: func P-ifiable iff nu X in P}, together with the fact that \(\{\R\}\cup\Cpt \subset \Lind\subset \RCpt\).
\end{proof}

%% file: section/4_kappa_Lind.tex
\Section[kappa-Lind]{Functorial \(\kappa\)-Lindel\"ofifiability of Discrete Spaces}

In this section, we study a relationship between the functorial \(\kappa\)-Lindel\"ofifiability of discrete spaces and properties of the cardinality of discrete spaces.

\begin{definition}
  Let \(\kappa\) and \(\lambda\) be cardinals, \(\alpha\) an ordinal, and \(X\) a topological space.
  \begin{enumerate}
    \item
    We shall use the notations \(\alpha + 1\), \(\omega_{\alpha}\), \(\aleph_{\alpha}\), \(\kappa^+\), and \(\cf(\alpha)\) as they are defined in \cite[Chapter 1, 7.10, 7.18, 10.17, 10.18, 10.30]{Kunen-set-intro}.
    Then, it is a well-known fact that \(\alpha + 1\) is compact.
    \item
    We shall say that \(X\) is \textbf{\(\kappa\)-Lindel\"of} if for any open covering \(\mcU\) of \(X\), there exists a subset \(\mcV\subset \mcU\) such that \(\bigcup \mcV = X\), and \(|\mcV|<\kappa\).
    \item
    We shall write \(\KLind\subset \Top\) for the full subcategory determined by the \(\kappa\)-Lindel\"of spaces.
    \item
    We shall write
    \[
      \up_{\kappa}X \dfn \bigcap \left\{
        X\subset Y\subset \beta X \,\middle|\, Y\in\KLind\,
      \right\} \,\subset\, \beta X.
    \]
  \end{enumerate}
\end{definition}

\begin{remark}\label{rem: kappa Lindify deg remarks}
  Note that the notion of an \(\aleph_0\)-Lindel\"of space (resp. an \(\aleph_1\)-Lindel\"of space) is equivalent to the notion of a compact space (resp. a Lindel\"of space).
  In particular, it holds that \(\beta X = \up_{\aleph_0}X\) and that \(\up X=\up_{\aleph_1} X\).
  Moreover, it holds that
  \[
    \beta X = \up_{\aleph_0} X \,\supset\, \up X = \up_{\aleph_1}X \,\supset\, \cdots \,\supset\, \up_{|X|^+} X = X.
  \]
\end{remark}

\begin{lemma}\label{lem: perfect inverse kappa Lind}
  Let \(\kappa\), \(\lambda\) be infinite cardinals and \(f:X\to Y\) a perfect map (cf. \cref{rmk: structure of realcompactification}).
  Then, if \(Y\) is \(\kappa\)-Lindel\"of, then \(X\) is also \(\kappa\)-Lindel\"of.
\end{lemma}

\begin{proof}
  Assume that \(Y\) is \(\kappa\)-Lindel\"of.
  Let \(\mcU\) be an open covering of \(X\) and \(y\in Y\) a point.
  Since \(f^{-1}(y)\) is compact, there exists a finite subset \(\mcU_y\subset \mcU\) such that \(f^{-1}(y)\subset \bigcup \mcU_y\).
  For any \(y\in \im(f)\), write \(V_y\dfn \bigcap_{U\in \mcU_y}(Y\setminus f(X\setminus U))\).
  Since \(f\) is closed, \(V_y\) is open.
  Moreover, since \(y\in V_y\), the family \(\mcV\dfn \left\{V_y\,\middle|\, y\in \im(f)\right\}\) is an open covering of \(Y\).
  Since \(Y\) is \(\kappa\)-Lindel\"of, there exists a subset \(\mcV_0\subset \mcV\) such that \(Y=\bigcup\mcV_0\), and, moreover, \(|\mcV_0| < \kappa\).
  Then, it follows immediately that the subset \(\mcU_0\dfn \bigcup\left\{\mcU_y\,\middle|\,y\in Y, V_y\in \mcV_0\right\}\) is an open covering of \(X\) such that \(|\mcU_0| < \kappa\).
  This implies that \(X\) is \(\kappa\)-Lindel\"of.
  This completes the proof of \cref{lem: perfect inverse kappa Lind}.
\end{proof}

\begin{corollary}\label{cor: kappa lindify}
  Let \(X\) be a topological space and \(\kappa\) an infinite cardinal.
  Then, \(X\) is functorial \(\kappa\)-Lindel\"ofifiable if and only if \(\up_{\kappa}X\) is \(\kappa\)-Lindel\"of.
\end{corollary}

\begin{proof}
  By \cref{lem: perfect inverse kappa Lind}, it holds that \(\tilde{\nu}_{\KLind}X = \up_{\kappa} X\).
  Thus, \cref{cor: kappa lindify} follows immediately from \cref{lem: nu X and naive}.
\end{proof}

Next, we study the functorial \(\kappa\)-Lindel\"ofifiability of a discrete space.

\begin{remark}\label{rmk: measurable def}
  Let \(\kappa\) is a cardinal.
  Recall that a filter \(U\) on a set \(X\) is \(\kappa\)-complete if for any family \(U_0\subset U\) such that \(|U_0| < \kappa\), it holds that \(\bigcap U_0\in U\).
  Note that a \(\{0,1\}\)-valued (\(\sigma\)-additive) measure \(\mu\) on a set \(X\) corresponds to an \(\aleph_0\)-complete ultrafilter \(\mcF_{\mu}\) as follows: for any \(U\subset X\), \(\mu(U)=1\) if and only if \(U\in \mcF_{\mu}\).

  Recall that a cardinal \(\kappa\) is \textbf{measurable} if \(\kappa\) is uncountable, and, moreover, there exists a \(\kappa\)-complete nonprincipal ultrafilter on \(\kappa\) (cf. \cite[Definition 10.3]{Jech-set-theory}).
  Note that this definition is different to the definition of the measurablility in \cite[Chapter 12]{GJ-rings-of-conti}.
  By \cite[Lemma 10.2]{Jech-set-theory}, the least cardinal that carries a non-trivial \(\{0,1\}\)-valued (\(\sigma\)-additive) measure is measurable.
  Conversely, if \(\kappa\) is measurable, then the \(\{0,1\}\)-valued measure on \(\kappa\) constructed from a \(\kappa\)-complete ultrafilter is (\(\sigma\)-additive and) non-trivial.
  Hence, the least cardinal that carries a non-trivial \(\{0,1\}\)-valued measure is equal to the least measurable cardinal.

  If \(\kappa\) admits a non-trivial \(\{0,1\}\)-valued (\(\sigma\)-additive) measure, then, for any \(\lambda > \kappa\), \(\lambda\) admits a non-trivial \(\{0,1\}\)-valued measure.
  Thus, by \cite[Theorem 12.2]{GJ-rings-of-conti}, a discrete space \(|X|\) is realcompact if and only if any cardinal less than \(|X|\) is non-measurable.
\end{remark}

In the remainder of the present section, we use the following notation:

\begin{definition}
  Let \(X\) be a discrete space.
  \begin{enumerate}
    \item
    Let \(U\subset X\) be a subset.
    We shall write
    \[O(U)\dfn \left\{\mcF\,\middle|\, \text{\(\mcF\) is an ultrafilter over \(X\), and, moreover, \(U\in \mcF\)}\, \right\}.\]
    \item
    We shall regard \(\beta X\) as the set of ultrafilters over \(X\) equipped with the topology generated by the family \(\{O(U)\,|\,U\subset X\}\).
    Then, we identify a point of \(X\) with a principal ultrafilter over \(X\).
  \end{enumerate}
\end{definition}

\begin{theorem}\label{lem: p in up kappa and completeness}
  Let \(\kappa\) be an infinite cardinal, \(X\) a discrete space, and \(\mcF\in \beta X \setminus X\) a point.
  Then, \(\mcF\in \up_{\kappa} X\) if and only if \(\mcF\) is \(\kappa\)-complete.
\end{theorem}

\begin{proof}
  First, we prove necessity.
  Assume that there exists \(\mcF_0\subset \mcF\) such that \(|\mcF_0| < \kappa\) and that \(\bigcap\mcF_0\not\in \mcF\).
  Write
  \[Y\dfn \bigcup\left\{ \overline{X\setminus F} \subset \beta X\,\middle|\, F\in \mcF_0\right\},\]
  where the closure is taken as a subset of \(\beta X\).
  Since \(\mcF\) is an ultrafilter over \(X\), for any \(F\in \mcF_0\), it holds that \(X\setminus F\not\in\mcF\).
  This implies that for any \(F\in \mcF_0\), \(\mcF\not\in\overline{X\setminus F}\).
  In particular, it holds that \(\mcF\not\in Y\).
  Since for each \(F\in \mcF_0\), \(\overline{X\setminus F}\subset \beta X\) is compact, and \(|\mcF_0| < \kappa\), we conclude that \(Y\) is \(\kappa\)-Lindel\"of.
  Thus, it holds that \(\mcF\not\in\up_{\kappa}X\).
  This completes the proof of necessity.

  Next, we prove sufficiency.
  Assume that there exists a \(\kappa\)-Lindel\"of subspace \(X\subset Y\subset \beta X\) such that \(\mcF\not\in Y\).
  For any element \(\mcG\in Y\), there exist open subsets \(\mcF\in U_{\mcG}\subset \beta X\) and \(\mcG\in V_{\mcG}\subset \beta X\) such that \(U_{\mcG}\cap V_{\mcG} = \emptyset\).
  Then, it holds that \(X\cap V_{\mcG}\not\in \mcF\).
  Since \(\mcF\) is an ultrafilter over \(X\), it holds that \(X\setminus V_{\mcG}\in \mcF\).
  Since \(Y\) is \(\kappa\)-Lindel\"of, there exists a subset \(Y_0\subset Y\) such that \(|Y_0|<\kappa\), and \(Y \subset \bigcup_{\mcG\in Y_0}V_{\mcG}\).
  Then, it holds that
  \[
    \bigcap_{\mcG\in Y_0}X\setminus V_{\mcG}
    \,=\, X\setminus \bigcup_{\mcG\in Y_0}V_{\mcG}
    \,\subset\, X\setminus Y \,=\, \emptyset.
  \]
  This implies that \(\mcF\) is not \(\kappa\)-complete.
  This completes the proof of \cref{lem: p in up kappa and completeness}.
\end{proof}

\begin{corollary}\label{lem: kappa Lind discrete}
  Let \(\kappa\) be an uncountable cardinal and \(X\) a discrete space.
  Then the following assertions hold:
  \begin{assertion}
    \item \label{ass: X is nonmeas lem: kappa Lind discrete}
    Assume that any cardinal less than or equal to \(|X|\) is non-measurable.
    Then, it holds that \(\up_{\kappa}X = X\).
    In particular, \(X\) is functorial \(\kappa\)-Lindel\"ofifiable if and only if \(|X| < \kappa\).
    \item \label{ass: kappa Lind equal X lem: kappa Lind discrete}
    Assume that any cardinal less than or equal to \(\kappa\) is non-measurable.
    Then, it holds that \(\up_{\kappa}X = \up X\).
    In particular, \(X\) is functorial \(\kappa\)-Lindel\"ofifiable if and only if \(\up X\) is \(\kappa\)-Lindel\"of.
  \end{assertion}
\end{corollary}

\begin{proof}
  \Cref{ass: X is nonmeas lem: kappa Lind discrete} follows immediately from \cite[Theorem 12.2]{GJ-rings-of-conti} and \cref{rem: kappa Lindify deg remarks}.
  \Cref{ass: kappa Lind equal X lem: kappa Lind discrete} follows immediately from \cite[Lemma 10.2]{Jech-set-theory}, \cref{lem: p in up kappa and completeness}, and the fact that any \(\kappa\)-complete ultrafilter over a subset of \(X\) induces a \(\kappa\)-complete ultrafilter over \(X\).
\end{proof}

\begin{corollary}\label{cor: exists measurable}
  If a discrete space \(X\) admits a functorial \(|X|\)-Lindel\"ofification, then \(|X|\) is measurable.
\end{corollary}

\begin{proof}
  Since \(X\) is not \(|X|\)-Lindel\"of, and \(X\) admits a functorial \(|X|\)-Lindel\"ofification, it holds that \(X\subsetneq \up_{|X|}X\).
  Thus, by \cref{lem: p in up kappa and completeness}, there exists a non-trivial \(|X|\)-complete ultrafilter over \(X\).
  In particular, \(|X|\) is measurable.
  This completes the proof of \cref{cor: exists measurable}.
\end{proof}

Next, we consider the converse implication of \cref{cor: exists measurable}.

\begin{definition}\label{def: strg cpt}
  We shall say that a cardinal \(\kappa\) is \textbf{strongly compact} if for any set \(S\), every \(\kappa\)-complete filter over \(S\) can be extended to a \(\kappa\)-complete ultrafilter over \(S\) (cf. \cite[Chapter 1, Proposition 4.1]{Kanamori-the-higher-infinite}).
\end{definition}

\begin{theorem}\label{thm: meas kappa Lind}
  Let \(\kappa\) be a cardinal.
  Then, the following assertions are equivalent:
  \begin{assertion}
    \item \label{ass: str cpt thm: meas kappa Lind}
    \(\kappa\) is strongly compact.
    \item \label{ass: admits kappa Lind thm: meas kappa Lind}
    Any discrete space admits a functorial \(\kappa\)-Lindel\"ofification.
  \end{assertion}
\end{theorem}

\begin{proof}
  First, we prove the implication ``\ref{ass: str cpt thm: meas kappa Lind} \(\Rightarrow\) \ref{ass: admits kappa Lind thm: meas kappa Lind}''.
  Assume that \(\kappa\) is strongly compact.
  Let \(X\) be a discrete space.
  By \cref{cor: kappa lindify}, to prove that \(X\) admits a functorial \(\kappa\)-Lindel\"ofification, it sufficies to prove that \(\up_{\kappa}X\) is \(\kappa\)-Lindel\"of.
  For each family of subsets \(\mcA\) of \(X\), consider the following condition
  \begin{itemize}
    \item[(\(\dagger_{\mcA}\))] 
    For any \(\kappa\)-complete ultrafilter \(\mcF\) over \(X\), there exists an element \(U\in \mcA\) such that \(U\in \mcF\).
  \end{itemize}
  Then, by \cref{lem: p in up kappa and completeness}, \(\mcA\) satisfies condition (\(\dagger_{\mcA}\)) if and only if \(\up_{\kappa}X\subset \bigcup_{A\in \mcA}O(A)\).
  Hence, to prove that \(\up_{\kappa}X\) is \(\kappa\)-Lindel\"of, it sufficies to prove that for any family of subsets \(\mcU\) of \(X\) that satisfies condition (\(\dagger_{\mcU}\)), there exists a subset \(\mcU'\subset \mcU\) of cardinality less than \(\kappa\) such that \(\mcU'\) satisfies condition (\(\dagger_{\mcU_0}\)).

  Let \(\mcU\) be a family of subsets of \(X\) that satisfies condition (\(\dagger_{\mcU}\)).
  Write
  \[
    \mcF\dfn \left\{ {\textstyle \bigcap_{U\in \mcU'}X\setminus U} \,\middle|\, \mcU'\subset \mcU, |\mcU'| < \kappa\right\}.
  \]
  Assume that any subset \(\mcU'\subset \mcU\) of cardinality less than \(\kappa\) does not satisfy condition (\(\dagger_{\mcU'}\)).
  Let \(\mcU'\subset \mcU\) be a subset such that \(|\mcU'| < \kappa\).
  Since \(\mcU'\) does not satisfy condition (\(\dagger_{\mcU'}\)), it follows from \cref{lem: p in up kappa and completeness} that there exists a \(\kappa\)-complete ultrafilter \(\mcF_{\mcU'}\in \up_{\kappa}X\setminus \bigcup_{U\in\mcU'}O(U)\) over \(X\).
  Then, it holds that \(\{X\setminus U\,|\, U\in \mcU'\}\subset \mcF_{\mcU'}\).
  Since \(|\mcU'| < \kappa\), and \(\mcF_{\mcU'}\) is \(\kappa\)-complete, this implies that \(\bigcap_{U\in\mcU'}X\setminus U\neq \emptyset\).
  In particular, \(\mcF\) is a \(\kappa\)-complete filter base over \(X\).
  Since \(\kappa\) is strongly compact, there exists a \(\kappa\)-complete ultrafilter \(\mcF^{\dagger}\) such that \(\mcF\subset \mcF^{\dagger}\).
  Since \(\mcU\) satisfies condition (\(\dagger_{\mcU}\)), this implies that \[\mcF^{\dagger}\in \beta X \setminus \bigcup\{O(U)\,|\,U\in \mcU\}\subset \beta X\setminus \up_{\kappa}X,\]
  in contradiction to the fact that \(\mcF^{\dagger}\) is \(\kappa\)-complete.
  Thus, there exists a subset \(\mcU_0\subset \mcU\) of cardinality less than \(\kappa\) such that \(\mcU_0\) satisfies condition (\(\dagger_{\mcU_0}\)).
  This completes the proof of the implication ``\ref{ass: str cpt thm: meas kappa Lind} \(\Rightarrow\) \ref{ass: admits kappa Lind thm: meas kappa Lind}''.

  Next, we prove the implication ``\ref{ass: admits kappa Lind thm: meas kappa Lind} \(\Rightarrow\) \ref{ass: str cpt thm: meas kappa Lind}''.
  Assume that any discrete space admits a functorial \(\kappa\)-Lindel\"ofification.
  Let \(X\) be a set and \(\mcF\) a \(\kappa\)-complete filter over \(X\).
  We regard \(X\) as a discrete topological space.
  Then, \(X\) admits a functorial \(\kappa\)-Lindel\"ofification.
  Hence, by \cref{cor: kappa lindify}, \(\up_{\kappa}X\) is \(\kappa\)-Lindel\"of.

  For any \(F\in \mcF\), write \(\tilde{F}\dfn \overline{F}\cap \up_{\kappa}X\subset \up_{\kappa}X\), where the closure \(\overline{F}\) is taken as a subset of \(\beta X\).
  Since \(\mcF\) is \(\kappa\)-complete, for any \(\mcF_0\subset \mcF\) such that \(|\mcF_0| < \kappa\), it holds that \(\bigcap_{F\in \mcF_0}\tilde{F} \supset \bigcap\mcF_0 \in \mcF\), i.e., \(\bigcap_{F\in \mcF_0}\tilde{F}\neq \emptyset\).
  Since \(\up_{\kappa}X\) is \(\kappa\)-Lindel\"of, it holds that \(\bigcap_{F\in \mcF}\tilde{F} \neq \emptyset\).
  Let \(\mcF^{\dagger}\in \bigcap_{F\in \mcF}\tilde{F}\) be an element.
  Then, it holds that \(\mcF\subset \mcF^{\dagger}\).
  Moreover, by \cref{lem: p in up kappa and completeness}, \(\mcF^{\dagger}\) is a \(\kappa\)-complete ultrafilter over \(X\).
  Thus, we conclude that \(\kappa\) is strongly compact.
  This completes the proof of \cref{thm: meas kappa Lind}.
\end{proof}

\begin{corollary}\label{cor: str cpt Lind}
  If the cardinality of a discrete space \(X\) is strongly compact, then \(X\) admits a functorial \(|X|\)-Lindel\"ofification.
\end{corollary}

\begin{proof}
  \Cref{cor: str cpt Lind} follows immediately from \cref{thm: meas kappa Lind}.
\end{proof}


%% file: section/5_example.tex
\Section{Examples}

In this final section, we give some examples of functorial \(\kappa\)-Lindel\"ofifiable spaces.
For any topological space \(X\) and any closed subspace \(A\subset X\), we shall write \(X/A\) for the quotient topological space of \(X\) by the equivalent relation \((A\times A)\cup \Delta_X \subset X\times X\), where \(\Delta_X\subset X\times X\) is the diagonal subset.
Since we assume that \(X\) is completely regular, \(X/A\) is Hausdorff.

We shall introduce the following cardinal functions:

\begin{definition}
  Let \(X\) be a topological space.
  \begin{enumerate}
    \item
    The \textbf{Lindel\"of degree} of \(X\), denoted \(L(X)\), is defined as the smallest infinite cardinal \(\kappa\) such that every open cover of \(X\) has a subcollection of cardinality \(\leq \kappa\) which covers \(X\) (cf. \cite[Chapter 1, {\S} 3]{Kunen-Handbook-set-top}).
    \item
    We shall write
    \[
      oL(X)\dfn \min\left\{ \kappa \,\middle|\,X = \up_{\kappa^+}X\,
      \right\} + \omega_0.
    \]
    We shall way that \(oL(X)\) is the \textbf{outer Lindel\"of degree} of \(X\).
  \end{enumerate}
\end{definition}

\begin{remark}
  Let \(X\) be a topological space.
  \begin{enumerate}
    \item
    One can verify easily that
    \[L(X) =  \min\left\{ \kappa \,\middle|\,\text{\(X\) is \(\kappa^+\)-Lindel\"of}\,
    \right\} + \omega_0.\]
    In particular, for a cardinal \(\kappa\), \(X\) is \(\kappa\)-Lindel\"of if and only if \(L(X) < \kappa\).
    \item
    It follows immediately that \(\aleph_0\leq oL(X) \leq L(X)\leq \max \{|X|, \aleph_0\}\).
    \item
    By \cref{rmk: structure of realcompactification}, \(X\) is realcompact if and only if \(oL(X) = \aleph_0\).
    Moreover, by \cref{rmk: measurable def}, if \(X\) is discrete, then \(|X|\) is non-measurable if and only if \(oL(X) = \aleph_0\).
    \item
    For any discrete space \(X\) such that \(2^{\aleph_0} < |X|\), if \(|X|\) is non-measurable, then it holds that \(2^{\aleph_0} = 2^{oL(X)\chi(X)} < |X|\), where \(\chi(X)\) is the \textit{character} of \(X\) (cf. \cite[Chapter 1, \S 3]{Kunen-Handbook-set-top}).
    Hence, the Arhangel'ski\u{\i}-type inequality for \(oL(-)\) and \(\chi(-)\) does not hold.
  \end{enumerate}
\end{remark}

For any linearly ordered set \((L, \leq)\) and elements \(a,b\in L\) such that \(a < b\), we shall write \([a,b]\dfn \left\{x\in L\,\middle|\, a\leq x \leq b\right\}\subset L\).
We regard \(L\) as a topoplogical space whose topology is generated by \(\left\{ [a,b]\setminus \{a,b\}\,\middle|\, a,b\in L\right\}\).

\begin{lemma}\label{lem: example lem}
  Let \(\kappa\) be an infinite cardinal and \(\alpha\) an infinite ordinal.
  Then, the following assertions hold:
  \begin{assertion}
    \item \label{ass: cf lambda lind lem: example lem}
    It holds that \(L(\alpha) = \cf(\alpha) + \omega_0\).
    Moreover, for any open covering \(\mcU\) of \(\alpha\), if \(|\mcU| < \cf(\alpha)\), then there exists a subset \(\mcV\subset \mcU\) such that \(|\mcV| < \aleph_0\).
    \item \label{ass: regular SC cpt lem: example lem}
    If \(\cf(\alpha) > \omega_0\), then \(\beta\alpha = \alpha + 1\).
    In particular, it holds that \(oL(\alpha) = \cf(\alpha) + \omega_0\).
    \item \label{ass: lambda 0 cpt lem: example lem}
    For any subset \(A\subset \kappa+1\), if \(\kappa\in A\), then \(L(A) < \kappa\).
  \end{assertion}
\end{lemma}

\begin{proof}
  First, we prove \cref{ass: cf lambda lind lem: example lem}.
  Let \(f:\cf(\alpha)\to \alpha\) be a map such that \(\im(f)\subset \alpha\) is unbounded and \(\mcU\) an open covering of \(\alpha\).
  Since \([0, -]\subset \alpha\) is compact, for any \(\gamma < \cf(\alpha)\), there exists a finite subcover \(\mcV_{\gamma}\subset \mcU\) such that \([0, f(\gamma)]\subset \bigcup \mcV_{\gamma}\).
  Then, \(\bigcup_{\gamma < \cf(\alpha)}\mcV_{\gamma}\subset \mcU\) is a subcover of \(\mcU\) whose cardinality is less than or equal to \(\cf(\alpha)\).
  This imples that \(\alpha\) is \(\cf(\alpha)^+\)-Lindel\"of.

  Assume that there exists a subset \(\mcV_0\) of the open covering \(\left\{ [0,\gamma]\,\middle|\,\gamma <\alpha\right\}\) of \(\alpha\) such that \(\kappa_0\dfn |\mcV_0| < \cf(\alpha)\) and that \(\alpha = \bigcup\mcV_0\).
  Let \(f_0:\kappa_0\tosim \mcV_0\) be a bijection.
  Then, since \(\alpha = \bigcup\mcV_0\), the map \(\kappa_0\to \alpha, \gamma \mapsto \max f_0(\gamma)\) is unbounded.
  This contradicts to our assumption that \(\kappa_0 < \cf(\alpha)\).
  Thus, \(\alpha\) is not \(\cf(\alpha)\)-Lindel\"of.
  In particular, it holds that \(L(\alpha) = \cf(\alpha)\).

  Let \(\mcW\) be an open covering of \(\alpha\) such that \(|\mcW| < \cf(\alpha)\).
  If for any \(W\in \mcW\), \(\alpha\setminus W\subset \alpha\) is unbounded, then, by \cite[Chapter 2, Lemma 6.8 (a)]{Kunen-set-intro}, \(\bigcap_{W\in \mcW}(\alpha \setminus W)\subset \alpha\) is also unbounded.
  This contradicts to our assumption that \(\alpha = \bigcup \mcW\).
  This implies that there exists \(W\in \mcW\) such that \(\alpha \setminus W\) is bounded.
  Since \([0,\sup(\alpha \setminus W)] \subset \bigcup \mcW\), and \([0,\sup(\alpha \setminus W)]\) is compact, there exists a finite subset \(\mcW_0\subset \mcW\) such that \([0,\sup(\alpha\setminus W)] \subset \bigcup \mcW_0\).
  Thus, it holds that \(\alpha = W \cup \bigcup\mcW_0\).
  This completes the proof of \cref{ass: cf lambda lind lem: example lem}.

  Next, we prove \cref{ass: regular SC cpt lem: example lem}.
  By \cref{ass: cf lambda lind lem: example lem}, \(\alpha\) is countably compact.
  Hence, \(\alpha\) is pseudocompact.
  Let \(f:\alpha\to \R\) be a continuous map.
  Then, \(f\) is bounded.
  To prove that \(\beta \alpha = \alpha + 1\), it sufficies to prove that there exists \(\gamma < \alpha\) such that for any \(\gamma < \gamma_0 < \alpha\), \(f(\gamma) = f(\gamma_0)\).
  Hence, we may assume without loss of generality that \(\im(f)\subset [0,1]\).
  Since \(\cf(\alpha) > \aleph_0\), it follows from \cite[Chapter 2, Lemma 6.8 (a)]{Kunen-set-intro} that for any \(n\in \N\), there exists a unique \(0\leq k(n) < 2^n\) such that \(f^{-1}([k(n)/2^n, (k(n)+1)/2^n))\subset \alpha\) is unbounded.
  Write \(a\) for the unique element \(\bigcap_{n\in \N}[k(n)/2^n, (k(n)+1)/2^n)\).
  Since \(\cf(\alpha) > \aleph_0\), there exists an ordinal \(\gamma < \alpha\) such that \(f([\gamma, \alpha)) = \{a\}\).
  This implies that \(\beta\alpha = \alpha + 1\).
  Moreover, by \cref{ass: cf lambda lind lem: example lem}, for any subspace \(\alpha \subset X\subset \beta \alpha\), if \(L(X) < \cf(\alpha)\), then \(X = \beta\alpha\).
  Hence, by \cref{ass: cf lambda lind lem: example lem}, it holds that \(\cf(\alpha)\leq oL(\alpha) \leq L(\alpha) = \cf(\alpha)\).
  This completes the proof of \cref{ass: regular SC cpt lem: example lem}.

  Next, we prove \cref{ass: lambda 0 cpt lem: example lem}.
  Let \(\mcU\) be an open covering of \(A\).
  Then, there exists an element \(U\in \mcU\) such that \(\kappa\in U\).
  Since \(|\kappa \setminus U| < \kappa\), there exists a subset \(\mcV\subset \mcU\) such that \(|\mcV| < \kappa\), and \(A \setminus U \subset \bigcup \mcV\).
  Then, \(\mcV \cup \{U\}\subset \mcU\) is a subset such that \(A\subset \bigcup(\mcV\cup\{U\})\), and \(|\mcV\cup\{U\}| < \kappa\).
  This completes the proof of \cref{lem: example lem}.
\end{proof}

\begin{definition}\label{def: Xd discrete}
  \
  \begin{enumerate}
    \item
    For any set \(X\), we shall write \(X_d\) for the discrete topological space obtained by the set \(X\).
    \item
    For any topological spaces \(X\) and \(Y\), we shall write \(X\sqcup Y\) for the disjoint union of \(X\) and \(Y\).
  \end{enumerate}
\end{definition}

\begin{example}\label{exam: d cup ord}
  Let \(\kappa\) and \(\lambda\) be (infinite) non-measurable cardinals.
  Then, the following assertions hold:
  \begin{assertion}
    \item \label{ass: oL lt L exam: d cup ord}
    Assume that \(\aleph_0 < \cf(\kappa) = \kappa < \lambda\).
    Then, it holds that \(oL(\kappa\sqcup\lambda_d) = \kappa < L(\kappa\sqcup\lambda_d) = \lambda\).
    \item \label{ass: oL eq L exam: d cup ord}
    Assume that \(\kappa < \cf(\lambda) = \lambda\).
    Then, it holds that \(oL(\kappa_d\sqcup\lambda) = L(\kappa_d\sqcup\lambda) = \lambda\), and, moreover, for any \(\kappa < \mu\leq \lambda\), \(\kappa_d\sqcup\lambda\) is functorial \(\mu\)-Lindel\"ofifiable but not functorial \(\kappa\)-Lindel\"ofifiable.
  \end{assertion}
\end{example}

\begin{proof}
  Since \(\beta(\kappa\sqcup\lambda_d) = (\kappa + 1)\sqcup \beta\lambda_d\), \cref{ass: oL lt L exam: d cup ord} and the equality \(oL(\kappa_d\sqcup\lambda) = L(\kappa_d\sqcup\lambda) = \lambda\) follow immediately from \cref{lem: kappa Lind discrete} \ref{ass: X is nonmeas lem: kappa Lind discrete} and \cref{lem: example lem} \ref{ass: cf lambda lind lem: example lem} \ref{ass: regular SC cpt lem: example lem}.
  Moreover, for any \(\aleph_0 < \mu\leq \lambda\), it holds that \(\up_{\mu}(\kappa_d\sqcup\lambda) = \kappa_d\sqcup(\lambda + 1)\).
  Thus, the last assertion of \cref{ass: oL eq L exam: d cup ord} follows.
  This completes the each assertions in \cref{exam: d cup ord}.
\end{proof}

Next, we consider classes smaller than \(\Lind\).

\begin{example}
  Write \(\LCLind\subset \Top\) for the full subcategory determined by the locally compact Lindel\"of spaces.
  The space of rational numbers \(\Q\subset \R\) (whose topology is induced from \(\R\)) is \(\sigma\)-compact but not locally compact.
  Thus, \(\Q\) is not functorial \(\LCLind\)-ifiable.
\end{example}

Finally, we give an example of functorial Lindel\"ofifiable space that is not functorial \(\sigma\)-compactifiable.

\begin{example}
  Write \(A_0\subset \omega_1\) for the set of countable successor ordinals and \(A\dfn A_0\cup\{\omega_1\}\subset \omega_1 + 1\).
  Then, since \(A_0\) is discrete, any compact subset of \(A\) is finite.
  In particular, \(A\) is Lindel\"of (cf. \cref{lem: example lem} \ref{ass: lambda 0 cpt lem: example lem}) but not \(\sigma\)-compact.
  Write
  \[
    X\dfn (\omega_0 \times \omega_1)\cup (\{\omega_0\}\times A)
    \,\subset\,
    (\omega_0+1)\times(\omega_1+1).
  \]
  Let \(n < \omega_0\) be an integer and \(f:\{n\}\times \omega_1 \to \R\) a continuous function.
  Since \(\beta \omega_1 = \omega_1 + 1\), there exists a continuous extention \(f_1:\{n\}\times (\omega_1 + 1) \to \R\) of \(f\).
  By applying Tieze's extention theorem to \((\omega_0+1)\times (\omega_1+1)\), there exists a continuous extension \(\tilde{f}: X \to \R\) of \(f_1\).
  This implies that \(\omega_0 \times (\omega_1 + 1)\subset \up X\).
  Since \(\omega_0 \times (\omega_1 + 1)\) is Lindel\"of, and \(A\) is Lindel\"of, it holds that
  \[
    \up X = (\omega_0 \times (\omega_1 + 1))\cup (\{\omega_0\}\times A).
  \]
  Since \(A\) is not \(\sigma\)-compact, this implies that \(X\) is functorial Lindel\"ofifiable but not functorial \(\sigma\)-compactifiable.
  \qed
\end{example}